\documentclass[12pt,twoside]{amsart}
\usepackage{amssymb}
\usepackage{amsmath}
\usepackage{xypic}

\dedicatory{Dedicated to Professor Shigefumi Mori}

\title{Vanishing theorems}
\author{Osamu Fujino} 
\date{2013/1/24, version 1.32}
\address{Department of Mathematics, Faculty of Science, 
Kyoto University, Kyoto 606-8502, Japan}
\email{fujino@math.kyoto-u.ac.jp}

\newcommand{\Spec}[0]{{\operatorname{Spec}}}

\newcommand{\Exc}[0]{{\operatorname{Exc}}}
\newcommand{\Supp}[0]{{\operatorname{Supp}}}
\newcommand{\Div}[0]{{\operatorname{Div}}}
\newcommand{\PerDiv}[0]{{\operatorname{PerDiv}}}
\newcommand{\Weil}[0]{{\operatorname{Weil}}}
\newcommand{\Gr}[0]{{\operatorname{Gr}}}
\newtheorem{thm}{Theorem}[section]
\newtheorem{lem}[thm]{Lemma}
\newtheorem{cor}[thm]{Corollary}
\newtheorem{prop}[thm]{Proposition}

\theoremstyle{definition}
\newtheorem{ex}[thm]{Example}
\newtheorem{defn}[thm]{Definition}
\newtheorem{rem}[thm]{Remark}
\newtheorem*{ack}{Acknowledgments} 

\newtheorem{say}[thm]{}
\newtheorem{step}{Step}

\begin{document}

\maketitle 

\begin{abstract}
We prove some injectivity, torsion-free, 
and vanishing theorems for 
simple normal crossing pairs. 
Our results heavily depend on the 
theory of mixed Hodge structures on cohomology groups with compact support. 
We also treat several 
basic properties of semi divisorial 
log terminal pairs. 
\end{abstract}

\tableofcontents
\section{Introduction}
In this paper, we prove some 
vanishing theorems for simple normal crossing pairs, which will play 
important roles in the study of higher dimensional algebraic varieties. 
We note that the notion of {\em{simple normal crossing pairs}} includes 
here the case when the ambient variety itself has several irreducible components with simple 
normal crossings. 
Theorem \ref{main} is a generalization of the works of several authors:~Kawamata, 
Viehweg, Koll\'ar, Esnault--Viehweg, Ambro, Fujino, and others (cf.~\cite{kollar1}, 
\cite{kmm}, \cite{ev}, \cite{kollar-sha}, 
\cite{ambro}, \cite{high}, \cite{fujino-mi}, 
\cite{fujino-on}, \cite{book}, 
\cite{fujino-fund}, and so on). 

\begin{thm}[{see Theorem \ref{8}}]\label{main}
Let $(Y, \Delta)$ be a simple normal crossing 
pair such that $\Delta$ is a boundary $\mathbb R$-divisor on $Y$. 
Let $f:Y\to X$ be a proper morphism to an algebraic variety $X$ and let $L$ be a Cartier 
divisor on $Y$ such that 
$L-(K_Y+\Delta)$ is $f$-semi-ample. 
\begin{itemize}
\item[(i)] every associated prime of $R^qf_*\mathcal O_Y(L)$ is the 
generic point of the $f$-image of some stratum of $(Y, \Delta)$. 
\item[(ii)] 
let $\pi:X\to V$ be a projective morphism to an algebraic variety $V$ such 
that $$L-(K_Y+\Delta)\sim _{\mathbb R}f^*H$$ for 
some $\pi$-ample $\mathbb R$-Cartier 
$\mathbb R$-divisor $H$ on $X$. 
Then $R^qf_*\mathcal O_Y(L)$ is $\pi_*$-acyclic, that is, 
$$R^p\pi_*R^qf_*\mathcal O_Y(L)=0$$ for every $p>0$ and $q\geq 0$. 
\end{itemize}
\end{thm}

When $X$ is a divisor on a smooth variety $M$, Theorem \ref{main} 
is contained in \cite{ambro} and 
plays crucial roles in the theory of 
quasi-log varieties. 
For the details, see \cite[Chapter 3]{book} and \cite{fujino-qlog}. 
When $X$ is quasi-projective, 
it is proved in \cite[Section 7]{fuji-fuji}. 
Here, we need no extra assumptions on $X$. 
Therefore, Theorem \ref{main} is new. 
The theory of resolution of singularities 
for {\em{reducible}} varieties has recently been developing 
(cf.~\cite{bierstone-milman} and \cite{bierstone-p}). 
It refines several vanishing theorems in \cite{book}. 
It is one of the main themes of this paper. 
We will give a proof of Theorem \ref{main} in Section \ref{sec3}. 
Note that we do not treat {\em{normal crossing varieties}}. 
We only discuss {\em{simple}} normal crossing varieties because the theory of resolution of 
singularities for reducible varieties works well only for {\em{simple}} normal crossing 
varieties.  
We note that the fundamental theorems for the log minimal model program 
for log canonical pairs can be proved without using the theory of quasi-log varieties (cf.~\cite{fujino-non} 
and \cite{fujino-fund}). 
The case when $Y$ is smooth in Theorem \ref{main} is sufficient for \cite{fujino-non} and \cite{fujino-fund}. 
For that case, see \cite{fujino-on} and \cite[Sections 5 and 6]{fujino-fund}. 
Our proof of Theorem \ref{main} 
heavily depends on the theory of mixed Hodge structures on cohomology 
groups with compact support. 

\begin{say}[Hodge theoretic viewpoint]\label{hodge}
Let $X$ be a projective simple normal crossing variety with $\dim X=n$. 
We are mainly interested in $H^\bullet (X, \omega_X)$ or 
$H^\bullet (X, \omega_X\otimes L)$ for some 
line bundle $L$ on $X$. 
By the theory of mixed Hodge structures, 
$$
\Gr_F^nH^\bullet (X, \mathbb C)\simeq H^{\bullet -n}(X, \nu_*\omega_{X^\nu}), 
$$ 
where $\nu:X^\nu\to X$ is the normalization, 
and 
$$
\Gr_F^0H^{\bullet}(X, \mathbb C)\simeq H^\bullet(X, \mathcal O_X). 
$$
Note that $F$ is the Hodge filtration on the natural mixed Hodge 
structure on $H^\bullet (X, \mathbb Q)$. 
Let $D$ be a simple normal crossing divisor on $X$. 
Then we obtain 
$$
\Gr_F^nH^\bullet (X\setminus D, \mathbb C)\simeq H^{\bullet -n}(X, \nu_*\omega_{X^\nu}\otimes 
\mathcal O_X(D))
$$ 
and 
$$
\Gr_F^0H^\bullet _c(X\setminus D, \mathbb C)\simeq H^\bullet (X, \mathcal O_X(-D)). 
$$
Note that 
$$
\Gr_F^0H^\bullet (X\setminus D, \mathbb C)\simeq H^\bullet (X, \mathcal O_X) 
$$ 
and that 
$$
H^{\bullet-n}(X, \nu_*\omega_{X^\nu}\otimes \mathcal O_X(D))
\not\simeq H^{\bullet -n}(X, \omega_X\otimes \mathcal O_X(D)). 
$$
We also note that 
$H^\bullet _c(X\setminus D, \mathbb Q)$ need not be the dual vector space of 
$H^{2n-\bullet}(X\setminus D, 
\mathbb Q)$ when $X$ is not smooth. 
In this setting, 
we are interested in $H^\bullet (X, \omega_X(D))$ or 
$H^\bullet (X, \omega_X(D)\otimes L)$. Therefore, we consider the natural mixed Hodge 
structure on $H^\bullet _c(X\setminus D, \mathbb C)$ and 
take the dual vector space of 
$$
\Gr_F^0H^\bullet _c(X\setminus D, \mathbb C)\simeq H^\bullet (X, \mathcal O_X(-D))
$$ 
by the Serre duality. 
Then we obtain $H^{n-\bullet} (X, \omega_X(D))$. 
We note that if $L$ is semi-ample then we can reduce the problem 
to the case when $L$ is trivial by the usual covering trick. 
The above observation is crucial for 
our treatment of the vanishing theorems and the semi-positivity theorems in \cite{book} and \cite{fuji-fuji}. 
In this paper, we do not discuss the Hodge theoretic part of vanishing and semi-positivity 
theorems. 
We prove Theorem \ref{main} by assuming 
the Hodge theoretic injectivity 
theorem:~Theorem \ref{2}. 
For the details of the Hodge theoretic 
part, see \cite{book}, \cite{fuji-fuji}, and \cite{fujino-inj}. 
\end{say}

The author learned the following example from Kento Fujita. 

\begin{ex}\label{rem-fujita} 
Let $X_1=\mathbb P^2$ and let $C_1$ be a line on $X_1$. 
Let $X_2=\mathbb P^2$ and let 
$C_2$ be a smooth conic on $X_2$. 
We fix an isomorphism $\tau:C_1\to C_2$. By gluing 
$X_1$ and $X_2$ along $\tau:C_1\to C_2$, we obtain a simple normal crossing surface 
$X$ such that $-K_X$ is ample (cf.~\cite{fujita}). 
We can check that $X$ can not be embedded into any smooth varieties as a simple normal crossing 
divisor. 
\end{ex}

Example \ref{rem-fujita} shows that 
Theorem \ref{main} is not covered by the results in \cite{ambro}, \cite{fujino-mi}, and \cite{book}. 

\begin{rem}[{cf.~\cite[Proposition 3.65]{book}}]
We can construct a proper simple normal crossing variety 
$X$ with the following property. Let $f:Y\to X$ be a proper 
morphism from a simple normal crossing variety $Y$ such that 
$f$ induces an isomorphism 
$f|_V:V\simeq U$ where 
$V$ (resp.~$U$) is a dense Zariski open subset 
of $Y$ (resp.~$X$) which 
contains the generic point of any stratum of $Y$ (resp.~$X$). 
Then $Y$ is non-projective. 
Therefore, we can not directly use Chow's lemma to reduce our 
main theorem (cf.~Theorem \ref{main}) to the quasi-projective case 
(cf.~\cite[Section 7]{fuji-fuji}). 
\end{rem}

There exists another standard approach to various Kodaira type vanishing theorems. 
It is an analytic method (see, for example, \cite{fujino-ana} and \cite{fujino-ana2}). 
At the present time, the relationship between our Hodge theoretic approach 
and the analytic method is not clear. 

We summarize the contents of this paper. 
In Section \ref{sec2}, we collect some basic definitions and results 
for the study of simple normal crossing varieties and divisors on them. 
Section \ref{sec3} is the main part of this paper. 
It is devoted to the study of injectivity, torsion-free, and vanishing theorems 
for simple normal crossing pairs. 
We note that we do not prove the Hodge theoretic injectivity theorem:~Theorem \ref{2}. 
We just quote it from \cite{book} (see also \cite{fujino-inj}). 
Section \ref{sec-sdlt} is an easy application of 
the vanishing theorem in Section \ref{sec3}. 
We prove the basic properties of semi divisorial log terminal pairs in the sense of 
Koll\'ar. 
In Section \ref{sec-semi}, we explain 
our new semi-positivity theorem, which is a generalization of 
the Fujita--Kawamata semi-positivity theorem, without proof. 
It depends on the theory of variations of mixed Hodge structures on cohomology groups with 
compact support and is related to the 
results obtained in Section \ref{sec3}. 
Anyway, the vanishing theorem and the semi-positivity theorem discussed in this paper follow 
from the theory of mixed Hodge structures on cohomology groups with compact support. 

For various applications of Theorem \ref{main} and related topics, 
see \cite{book}, \cite{fuji-fuji}, \cite{fujino-slc}, \cite{fujino-semi}, \cite{fujino-inj}, and so on. 

\begin{ack}
The author was partially supported by the Grant-in-Aid for Young Scientists 
(A) $\sharp$20684001 from JSPS. 
He would like to thank Professors Takeshi Abe, Taro Fujisawa, and Shunsuke Takagi 
for discussions and useful comments. He would also like to 
thank Professor J\'anos Koll\'ar for giving him a preliminary version of his book \cite{kollar-book}. 
The main part of this paper is a revised version of Section 5 of the author's 
unpublished preprint \cite{fujino-mi}. 
Although he tried to publish \cite{fujino-mi}, 
its importance could have been misunderstood by the referee in 2007, when almost 
all the minimal modelers were mainly interested in the big development 
by Birkar--Cascini--Hacon--M\textsuperscript{c}Kernan. When he wrote \cite{fujino-mi} in Nagoya, he was 
partially supported by the Grant-in-Aid for Young Scientists (A) $\sharp$17684001 from JSPS and 
by the Inamori Foundation. 
\end{ack}

We will work over $\mathbb C$, the complex number field, throughout this paper. 
But we note that, by using the Lefschetz principle, all the results in this paper 
hold over an algebraically closed field $k$ of characteristic zero. 

\section{Preliminaries}\label{sec2}
First, we quickly recall basic definitions of 
divisors. We note that we have to 
deal with reducible algebraic schemes in this paper. 
For details, see, for example, \cite[Section 2]{hartshorne} and \cite[Section 7.1]{liu}. 

\begin{say}\label{21ne}
Let $X$ be a noetherian scheme with structure sheaf $\mathcal O_X$ 
and let $\mathcal K_X$ be the sheaf of total quotient rings of $\mathcal O_X$. 
Let $\mathcal K^*_X$ denote the (multiplicative) sheaf of invertible 
elements in $\mathcal K_X$, and $\mathcal O^*_X$ the sheaf of invertible elements in $\mathcal O_X$. 
We note that $\mathcal O_X\subset \mathcal K_X$ and 
$\mathcal O^*_X\subset \mathcal K^*_X$. 
\end{say}

\begin{say}[Cartier, $\mathbb Q$-Cartier, and 
$\mathbb R$-Cartier divisors]
A {\em{Cartier divisor}} $D$ on $X$ is a global section of $\mathcal K^*_X/\mathcal O^*_X$, 
that is, $D$ is an element of 
$H^0(X, \mathcal K^*_X/\mathcal O^*_X)$. 
A {\em{$\mathbb Q$-Cartier $\mathbb Q$-divisor}} (resp.~{\em{$\mathbb R$-Cartier 
$\mathbb R$-divisor}}) is an element of 
$H^0(X, \mathcal K^*_X/\mathcal O^*_X)\otimes _{\mathbb Z}\mathbb Q$ 
(resp.~$H^0(X, \mathcal K^*_X/\mathcal O^*_X)\otimes _{\mathbb Z}\mathbb R$).  
\end{say}

\begin{say}[Linear, $\mathbb Q$-linear, and $\mathbb R$-linear equivalence] 
Let $D_1$ and $D_2$ be two $\mathbb R$-Cartier $\mathbb R$-divisors on $X$. 
Then $D_1$ is {\em{linearly}} (resp.~{\em{$\mathbb Q$-linearly}}, or 
{\em{$\mathbb R$-linearly}}) {\em{equivalent}} to $D_2$, denoted by 
$D_1\sim D_2$ (resp.~$D_1\sim_{\mathbb Q}D_2$, or 
$D_1\sim _{\mathbb R}D_2$) if $$D_1=D_2+\sum _{i=1}^k r_i(f_i)$$ such that 
$f_i\in \Gamma (X, \mathcal K^*_X)$ and $r_i\in \mathbb Z$ 
(resp.~$r_i\in \mathbb Q$, or $r_i\in \mathbb R$) for 
every $i$. We note that 
$(f_i)$ is a {\em{principal Cartier divisor}} 
associated to $f_i$, 
that is, the image of $f_i$ by $\Gamma (X, \mathcal K^*_X)\to 
\Gamma (X, \mathcal K^*_X/\mathcal O^*_X)$.  
Let $f:X\to Y$ be a morphism. 
If there is an $\mathbb R$-Cartier $\mathbb R$-divisor 
$B$ on $Y$ such that 
$D_1\sim _{\mathbb R}D_2+f^*B$, then $D_1$ 
is said to be {\em{relatively $\mathbb R$-linearly equivalent}} to $D_2$. 
It is denoted by $D_1\sim _{\mathbb R, f}D_2$. 
\end{say}

\begin{say}[Supports]
Let $D$ be a Cartier divisor on $X$. 
The {\em{support}} of $D$, denoted by $\Supp D$, is 
the subset of $X$ consisting of points $x$ such that 
a local equation for $D$ is not in $\mathcal O^*_{X, x}$. The 
support of $D$ is a closed subset of $X$. 
\end{say}

\begin{say}[Weil divisors, $\mathbb Q$-divisors, and $\mathbb R$-divisors] 
Let $X$ be an equi-dimensional reduced separated 
algebraic scheme. We note that 
$X$ is not necessarily regular in codimension one. 
A {\em{\em{(}}Weil{\em{)}} divisor} 
$D$ on $X$ is a finite formal sum 
$$
\sum _{i=1}^nd_i D_i
$$
where $D_i$ is an irreducible reduced closed subscheme of $X$ of pure 
codimension one and $d_i$ is an integer for every $i$ such that 
$D_i\ne D_j$ for $i\ne j$.  

If $d_i\in \mathbb Q$ (resp.~$d_i\in \mathbb R$) for every $i$, 
then $D$ is called a {\em{$\mathbb Q$-divisor}} (resp.~{\em{$\mathbb R$-divisor}}). 
We define the {\em{round-up}} $\lceil D\rceil=\sum _{i=1}^r\lceil d_i\rceil D_i$ 
(resp.~the {\em{round-down}} $\lfloor D\rfloor =\sum_{i=1}^r\lfloor d_i\rfloor D_i$), 
where for every real number $x$, $\lceil x\rceil$ (resp.~$\lfloor x\rfloor$) is the integer 
defined by $x\leq \lceil x\rceil<x+1$ (resp.~$x-1<\lfloor x\rfloor \leq x$). 
The {\em{fractional part}} $\{D\}$ of 
$D$ denotes $D-\lfloor D\rfloor$. 
We define $D^{<1}=\sum _{d_i<1}d_i D_i$ and so on. 
We call $D$ a {\em{boundary}} 
$\mathbb R$-divisor if $0\leq d_i\leq 1$ for every $i$. 
\end{say}

Next, we recall the definition of {\em{simple normal crossing 
pairs}}. 

\begin{defn}[Simple normal crossing pairs] 
We say that the pair $(X, D)$ is {\em{simple normal crossing}} at 
a point $a\in X$ if $X$ has a Zariski open neighborhood $U$ of $a$ that can be embedded in a smooth 
variety 
$Y$, 
where $Y$ has regular system of parameters $(x_1, \cdots, x_p, y_1, \cdots, y_r)$ at 
$a=0$ in which $U$ is defined by a monomial equation 
$$
x_1\cdots x_p=0
$$ 
and $$
D=\sum _{i=1}^r \alpha_i(y_i=0)|_U, \quad  \alpha_i\in \mathbb R. 
$$ 
We say that $(X, D)$ is a {\em{simple normal crossing pair}} if it is simple normal crossing at every point of $X$. 
If $(X, 0)$ is a simple normal crossing pair, then $X$ is called a {\em{simple normal crossing 
variety}}. If $X$ is a simple normal crossing variety, then $X$ has only Gorenstein singularities. 
Thus, it has an invertible dualizing sheaf $\omega_X$. 
Therefore, we can define the {\em{canonical divisor $K_X$}} such that 
$\omega_X\simeq \mathcal O_X(K_X)$ (cf.~\cite[Section 7.1 Corollary 1.19]{liu}). 
It is a Cartier divisor on $X$ and is well-defined up to linear equivalence. 
\end{defn}

We note that a simple normal crossing pair is called a {\em{semi-snc pair}} 
in \cite[Definition 1.9]{kollar-book}. 

\begin{defn}[Strata and permissibility]\label{stra} 
Let $X$ be a simple normal crossing variety and let $X=\bigcup _{i\in I}X_i$ be the 
irreducible decomposition of $X$. 
A {\em{stratum}} of $X$ is an irreducible component of $X_{i_1}\cap \cdots \cap X_{i_k}$ for some 
$\{i_1, \cdots, i_k\}\subset I$. 
A Cartier divisor $D$ on $X$ is {\em{permissible}} if $D$ contains no strata of $X$ in its support. 
A finite $\mathbb Q$-linear (resp.~$\mathbb R$-linear) combination of permissible 
Cartier divisors is called a {\em{permissible $\mathbb Q$-divisor}} (resp.~{\em{$\mathbb R$-divisor}}) on 
$X$. 
\end{defn}

\begin{say}
Let $X$ be a simple normal crossing variety. 
Let $\PerDiv (X)$ be the abelian group generated by permissible 
Cartier divisors on $X$ and let $\Weil (X)$ be the abelian group 
generated by Weil divisors on $X$. 
Then we can define natural injective homomorphisms of abelian groups 
$$
\psi:\PerDiv (X)\otimes _{\mathbb Z}\mathbb K\to \Weil (X)\otimes _{\mathbb Z}\mathbb K 
$$ 
for $\mathbb K=\mathbb Z$, $\mathbb Q$, and $\mathbb R$. 
Let $\nu:\widetilde X\to X$ be the normalization. 
Then we have the following commutative diagram. 
$$
\xymatrix{
\Div (\widetilde X)\otimes _{\mathbb Z}\mathbb K 
\ar[r]^{\sim}_{\widetilde \psi}& \Weil (\widetilde X)\otimes 
_{\mathbb Z}\mathbb K\ar[d]^{\nu_*} \\
\PerDiv (X)\otimes _{\mathbb Z}\mathbb K\ar[r]_{\psi}
\ar[u]^{\nu^*}&\Weil (X)\otimes _{\mathbb Z}\mathbb K
}
$$ 
Note that $\Div (\widetilde X)$ is the abelian group generated by 
Cartier divisors on $\widetilde X$ and 
that $\widetilde \psi$ is an isomorphism 
since $\widetilde X$ is smooth. 

By $\psi$, every permissible Cartier (resp.~$\mathbb Q$-Cartier 
or $\mathbb R$-Cartier) divisor can be considered 
as a Weil divisor (resp.~$\mathbb Q$-divisor or $\mathbb R$-divisor). 
Therefore, various operations, 
for example, $\lfloor D\rfloor$, $D^{<1}$, and so on, 
make sense for a permissible 
$\mathbb R$-Cartier $\mathbb R$-divisor $D$ on $X$. 
\end{say}

We note the following easy example. 

\begin{ex}
Let $X$ be a simple normal crossing variety 
in $\mathbb C^3=\Spec \mathbb C[x, y, z]$ defined 
by $xy=0$. 
We set $D_1=(x+z=0)\cap X$ and 
$D_2=(x-z=0)\cap X$. 
Then $D=\frac{1}{2} D_1+\frac{1}{2}D_2$ is a permissible 
$\mathbb Q$-Cartier $\mathbb Q$-divisor 
on $X$. 
In this case, $\lfloor D\rfloor=(x=z=0)$ on $X$. 
Therefore, $\lfloor D\rfloor$ is not a Cartier divisor on $X$. 
\end{ex}

\begin{defn}[Simple normal crossing divisors]
Let $X$ be a simple normal crossing variety and 
let $D$ be a Cartier divisor on $X$. 
If $(X, D)$ is a simple normal crossing pair and $D$ is reduced, 
then $D$ is called a {\em{simple normal crossing divisor}} on $X$. 
\end{defn}

\begin{rem}\label{211ne}
Let $X$ be a simple normal crossing variety and let $D$ be a 
$\mathbb K$-divisor on $X$ where 
$\mathbb K=\mathbb Q$ or $\mathbb R$. 
If $\Supp D$ is a simple normal crossing divisor on $X$ and $D$ is $\mathbb K$-Cartier, 
then $\lfloor D\rfloor$ and $\lceil D\rceil$ (resp.~$\{D\}$, $D^{<1}$, 
and so on) are Cartier (resp.~$\mathbb K$-Cartier) divisors on $X$ (cf.~\cite[Section 8]{bierstone-p}). 
\end{rem}

The following lemma is easy but important. 

\begin{lem}\label{7}
Let $X$ be a simple normal crossing variety and 
let $B$ be a permissible $\mathbb R$-Cartier $\mathbb R$-divisor on 
$X$ such that $\lfloor B\rfloor=0$. 
Let $A$ be a Cartier divisor on $X$. 
Assume that 
$A\sim _{\mathbb R}B$. Then 
there exists a permissible $\mathbb Q$-Cartier $\mathbb Q$-divisor 
$C$ on $X$ such that $A\sim _{\mathbb Q}C$, $\lfloor C\rfloor =0$, and 
$\Supp C=\Supp B$. 
\end{lem}
\begin{proof}
We can write $B=A+\sum _{i=1}^k r_i (f_i)$, 
where $f_i\in \Gamma (X, \mathcal K^* _X)$ and 
$r_i\in \mathbb R$ for every $i$. 
Here, $\mathcal K_X$ is the sheaf of 
total quotient rings of $\mathcal O_X$ (see \ref{21ne}). 
Let $P\in X$ be a scheme theoretic point corresponding to 
some stratum of $X$. 
We consider the following affine map 
$$
\mathbb K^k\to H^0(X_P, \mathcal K^*_{X_P}/\mathcal O^{*}_{X_P})\otimes _{\mathbb Z} \mathbb K
$$ 
given by $(a_1, \cdots, a_k)\mapsto A+\sum _{i=1}^k a_i (f_i)$, where 
$X_P=\Spec \mathcal O_{X, P}$ and $\mathbb K=\mathbb Q$ or $\mathbb R$. 
Then we can check that 
$$
\mathcal P=\{(a_1, \cdots, a_k)\in \mathbb R^k\,|\, A+\sum _i a_i (f_i)\ 
\text{is permissible}\}\subset \mathbb R^k
$$ 
is an affine subspace of $\mathbb R^k$ defined over $\mathbb Q$. 
Therefore, we see that 
$$
\mathcal S=\{(a_1, \cdots, a_k)\in 
\mathcal P\, |\, \Supp (A+\sum _i a_i (f_i))\subset \Supp B\}\subset 
\mathcal P 
$$ is an affine subspace of $\mathbb R^k$ defined over $\mathbb Q$. 
Since $(r_1, \cdots, r_k)\in \mathcal S$, we know 
that $\mathcal S\ne \emptyset$. 
We take a point $(s_1, \cdots, s_k)\in \mathcal S \cap \mathbb Q^k$ which 
is general in $\mathcal S$ and sufficiently close to 
$(r_1, \cdots, r_k)$ and set $C=A+\sum _{i=1}^k s_i (f_i)$. By construction, 
$C$ is a permissible $\mathbb Q$-Cartier $\mathbb Q$-divisor such that 
$C\sim _{\mathbb Q}A$, $\lfloor C\rfloor =0$, and 
$\Supp C=\Supp B$. 
\end{proof}

We need the following important definition in Section \ref{sec3}. 

\begin{defn}[Strata and permissibility for pairs]
Let $(X, D)$ be a simple normal crossing pair such that 
$D$ is a boundary $\mathbb R$-divisor 
on $X$. 
Let $\nu:X^\nu \to X$ be the normalization. 
We define $\Theta$ by the formula 
$$
K_{X^\nu}+\Theta=\nu^*(K_X+D). 
$$ 
Then a {\em{stratum}} of $(X, D)$ is an irreducible component of $X$ or the $\nu$-image 
of a log canonical center of $(X^\nu, \Theta)$. 
We note that $(X^\nu, \Theta)$ is log canonical.  
When $D=0$, 
this definition is compatible with Definition \ref{stra}. 
An $\mathbb R$-Cartier $\mathbb R$-divisor $B$ on $X$ is {\em{permissible with 
respect to $(X, D)$}} if $B$ contains no strata of $(X, D)$ in its support. 
If $B$ is a permissible $\mathbb R$-Cartier $\mathbb R$-divisor 
with respect to $(X, D)$, 
then we can easily check that 
$$
B=\sum _i b_i B_i
$$
where $B_i$ is a permissible Cartier divisor with respect to $(X, D)$ and $b_i\in \mathbb R$ for 
every $i$ (cf.~Proof of Lemma \ref{7}).  
\end{defn}

For the reader's convenience, we recall Grothendieck's Quot scheme. 
For the details, see, for example, \cite[Theorem 5.14]{fga} and 
\cite[Section 2]{ak}. 
We will use it in the proof of the main theorem:~Theorem \ref{8}. 

\begin{thm}[Grothendieck]\label{quot-s} 
Let $S$ be a noetherian scheme, 
let $\pi: X\to S$ be a projective morphism, and 
let $L$ be a relatively very ample line bundle on $X$. 
Then for any coherent $\mathcal O_X$-module $E$ 
and any polynomial $\Phi\in \mathbb Q[\lambda]$, 
the functor $\mathfrak{Quot}^{\Phi, L}_{E/X/S}$ 
is representable by a projective $S$-scheme 
$\mathrm{Quot}^{\Phi, L}_{E/X/S}$. 
\end{thm}

\section{Vanishing theorems}\label{sec3}

Let us start with the following injectivity theorem (cf.~\cite[Proposition 3.2]{fujino-mi} and 
\cite[Proposition 2.23]{book}). 
The proof of Theorem \ref{2} in \cite{book} is purely Hodge theoretic. 
It uses the theory of mixed Hodge structures on cohomology groups with compact support 
(cf.~\ref{hodge}). 
For the details, see \cite[Chapter 2]{book} and \cite{fujino-inj}. 

\begin{thm}[Hodge theoretic injectivity theorem]\label{2}
Let $(X, S+B)$ be a simple normal crossing 
pair such that $X$ is proper, $S+B$ is a boundary $\mathbb R$-divisor, 
$S$ is reduced, and 
$\lfloor B\rfloor =0$. 
Let $L$ be a Cartier divisor on $X$ and let $D$ be an 
effective Cartier divisor whose support is contained 
in $\Supp B$. 
Assume that 
$L\sim_{\mathbb R}K_X+S+B$. Then 
the natural homomorphisms 
$$
H^q(X, \mathcal O_X(L))\to H^q(X, \mathcal O_X(L+D)), 
$$ 
which are induced by the inclusion 
$\mathcal O_X\to \mathcal O_X(D)$, 
are injective for all $q$. 
\end{thm}

\begin{rem} 
In \cite{fujino-inj}, we prove a slight generalization of Theorem \ref{2}. However, Theorem \ref{2} is sufficient 
for the proof of Theorem \ref{5.1} below. 
\end{rem}

The next lemma is an easy generalization of the vanishing theorem of Reid--Fukuda 
type for simple normal crossing pairs, which is a very special case of Theorem \ref{8} (i). 
However, we need Lemma \ref{rf} for our proof of Theorem \ref{8}.  

\begin{lem}[Relative vanishing lemma]\label{rf} 
Let $f:Y\to X$ be a proper morphism 
from a simple normal crossing 
pair $(Y, \Delta)$ to an algebraic variety $X$ 
such that $\Delta$ is a boundary 
$\mathbb R$-divisor on $Y$. 
We assume that $f$ is an isomorphism 
at the generic point of any stratum of the pair $(Y, \Delta)$. 
Let $L$ be a Cartier divisor on $Y$ such 
that $L\sim _{\mathbb R, f}K_Y+\Delta$. 
Then $R^qf_*\mathcal O_Y(L)=0$ for every $q>0$. 
\end{lem} 

\begin{proof}
By shrinking $X$, we may assume that $L\sim _{\mathbb R}K_Y+\Delta$.  
By applying Lemma \ref{7} to $\{\Delta\}$, we may further assume that 
$\Delta$ is a $\mathbb Q$-divisor and 
$L\sim _{\mathbb Q}K_Y+\Delta$. 
\setcounter{step}{0}
\begin{step}\label{1ne}
We assume that $Y$ is irreducible. 
In this case, $L-(K_Y+\Delta)$ is 
nef and log big over $X$ with respect to the pair $(Y, \Delta)$, that is, 
$L-(K_Y+\Delta)$ is nef and big over $X$ and $(L-(K_Y+\Delta))|_W$ is 
big over $f(W)$ for every stratum $W$ of the pair $(Y, \Delta)$. 
Therefore, $R^qf_*\mathcal O_Y(L)=0$ for every $q>0$ 
by the vanishing theorem of Reid--Fukuda type (see, for example, \cite[Lemma 4.10]{book}). 
\end{step}
\begin{step}
Let $Y_1$ be an irreducible component of $Y$ and let $Y_2$ be the union 
of the other irreducible components of $Y$. 
Then we have a short exact sequence 
$$0\to \mathcal O_{Y_1}(-Y_2|_{Y_1})\to 
\mathcal O_Y\to \mathcal O_{Y_2}\to 0.$$ 
We set $L'=L|_{Y_1}-{Y_2}|_{Y_1}$. 
Then we have a short exact sequence 
$$0\to \mathcal O_{Y_1}(L')\to 
\mathcal O_Y(L)\to \mathcal O_{Y_2}(L|_{Y_2})\to 0$$ 
and $L'\sim _{\mathbb Q}K_{Y_1}+\Delta|_{Y_1}$. 
On the other hand, we can check 
that $$L|_{Y_2}\sim 
_{\mathbb Q}K_{Y_2}+{Y_1}|_{Y_2}+\Delta|_{Y_2}.$$ 
We have already known that $R^qf_*\mathcal O_{Y_1}(L')=0$ 
for every $q>0$ by Step \ref{1ne}. 
By the induction on the number of 
the irreducible components of 
$Y$, we have $R^qf_*\mathcal O_{Y_2}(L|_{Y_2})=0$ 
for every $q>0$. 
Therefore, $R^qf_*\mathcal O_Y(L)=0$ for every 
$q>0$ by the exact sequence: 
$$\cdots 
\to R^qf_*\mathcal O_{Y_1}(L')\to 
R^qf_*\mathcal O_Y(L)\to R^qf_*\mathcal O_{Y_2}
(L|_{Y_2})\to \cdots. $$
\end{step} 
So, we finish the proof of Lemma \ref{rf}. 
\end{proof}

Although Lemma \ref{rf} is a very easy generalization of the relative Kawamata--Viehweg 
vanishing theorem, it is sufficiently powerful for the study of reducible 
varieties once we combine it with the recent results 
in \cite{bierstone-milman} and \cite{bierstone-p}. 
In Section \ref{sec-sdlt}, 
we will see an application of Lemma \ref{rf} for the study of semi divisorial log 
terminal pairs.

It is the time to state the main injectivity theorem 
for simple normal crossing pairs. 
It is a direct application of Theorem \ref{2}. 
Our formulation of Theorem \ref{5.1} 
is indispensable for the proof of our main theorem:~Theorem \ref{8}. 

\begin{thm}[Injectivity theorem for simple normal crossing pairs]\label{5.1} 
Let $(X, \Delta)$ be a simple normal crossing 
pair such that $X$ is proper and that $\Delta$ is a boundary 
$\mathbb R$-divisor on $X$. 
Let $L$ be a Cartier 
divisor on $X$ and let $D$ be an effective Cartier divisor that is permissible 
with 
respect to $(X, \Delta)$. 
Assume the following conditions. 
\begin{itemize}
\item[(i)] $L\sim _{\mathbb R}K_X+\Delta+H$, 
\item[(ii)] $H$ is a semi-ample $\mathbb R$-Cartier 
$\mathbb R$-divisor, and 
\item[(iii)] $tH\sim _{\mathbb R} D+D'$ for some 
positive real number $t$, where 
$D'$ is an effective $\mathbb R$-Cartier 
$\mathbb R$-divisor that is permissible with respect to $(X, \Delta)$. 
\end{itemize}
Then the homomorphisms 
$$
H^q(X, \mathcal O_X(L))\to H^q(X, \mathcal O_X(L+D)), 
$$ 
which are induced by the natural inclusion 
$\mathcal O_X\to \mathcal O_X(D)$, 
are injective for all $q$. 
\end{thm}

\begin{rem}
For the definition and the basic properties of semi-ample $\mathbb R$-Cartier $\mathbb R$-divisors, 
see \cite[Definition 4.11, Lemma 4.13, and Lemma 4.14]{fujino-fund}. 
\end{rem}

\begin{proof}[Proof of {\em{Theorem \ref{5.1}}}] 
We set $S=\lfloor \Delta\rfloor$ and $B=\{\Delta\}$ throughout this proof. 
We obtain a projective birational 
morphism $f:Y\to X$ from 
a simple normal crossing 
variety $Y$ such that $f$ is an isomorphism 
over $X\setminus \Supp (D+D'+B)$, and that the 
union of the support of $f^*(S+B+D+D')$ and the 
exceptional locus of $f$ has a simple normal crossing support on $Y$ 
(cf.~\cite[Theorem 1.5]{bierstone-p}). 
Let $B'$ be the strict transform of $B$ on $Y$. 
We may assume that $\Supp B'$ is disjoint from 
any strata of $Y$ that are 
not irreducible components of $Y$ 
by taking blowing-ups. 
We write $$K_Y+S'+B'=f^*(K_X+S+B)+E, $$  
where $S'$ is the strict transform of $S$ and 
$E$ is $f$-exceptional. 
By the construction of $f:Y\to X$, 
$S'$ is Cartier and $B'$ is $\mathbb R$-Cartier. 
Therefore, $E$ is also $\mathbb R$-Cartier. 
It is easy to see that $E_+=\lceil E\rceil \geq 0$. 
We set $L'=f^*L+E_+$ and $E_-=E_+-E\geq 0$. 
We note that $E_+$ is Cartier and $E_-$ is 
$\mathbb R$-Cartier 
because $\Supp E$ is simple normal crossing on 
$Y$ (cf.~Remark \ref{211ne}). 
Since $f^*H$ is an $\mathbb R_{>0}$-linear 
combination of semi-ample Cartier 
divisors, we can write $f^*H\sim _{\mathbb R}
\sum _i a_i H_i$, where 
$0< a_i <1$ and $H_i$ is a general 
Cartier divisor on $Y$ for every $i$. 
We set 
$$B''=B'+E_-+\frac{\varepsilon}{t} 
f^*(D+D')+(1-\varepsilon) \sum _i a_i H_i$$ for 
some $0<\varepsilon \ll 1$. 
Then $L'\sim _{\mathbb R}K_Y+S'+B''$. 
By the construction, $\lfloor B''\rfloor =0$, the 
support of $S'+B''$ is simple normal crossing 
on $Y$, and $\Supp B''\supset \Supp f^*D$. 
So, Theorem \ref{2} implies that 
the homomorphisms 
$$H^q(Y, \mathcal O_Y(L'))\to H^q(Y, \mathcal O_Y(L'+f^*D))$$ are 
injective for all $q$.
By Lemma \ref{rf}, $R^qf_*\mathcal O_Y(L')=0$ for 
every $q>0$ and it is easy to see that $f_*\mathcal O_Y(L')\simeq 
\mathcal O_X(L)$. By the 
Leray spectral sequence, 
the homomorphisms $$H^q(X, \mathcal O_X(L))\to 
H^q(X, \mathcal O_X(L+D))$$ are injective for all $q$. 
\end{proof}

\begin{lem}\label{comp}
Let $f:Z\to X$ be a proper morphism 
from a simple normal crossing pair $(Z, B)$ to an algebraic variety $X$. 
Let $\overline X$ be a projective variety such that $\overline 
X$ contains $X$ as a Zariski dense open subset. 
Then there exist a proper simple normal crossing pair 
$(\overline Z, \overline B)$ that is a compactification of $(Z, B)$ 
and a proper morphism $\overline f: \overline Z\to 
\overline X$ that compactifies $f:Z\to X$. 
Moreover, $\overline Z\setminus Z$ is a divisor on $\overline Z$, $\Supp \overline B\cup 
\Supp (\overline Z\setminus Z)$ is a simple normal 
crossing divisor on $\overline Z$, and 
$\overline Z\setminus Z$ has no common 
irreducible components with $\overline B$. 
We note that we can make $\overline B$ a $\mathbb K$-Cartier 
$\mathbb K$-divisor on $\overline {Z}$ when so is $B$ on $Z$, 
where $\mathbb K$ is $\mathbb Z$, $\mathbb Q$, or $\mathbb R$. 
When $f$ is projective, we can make $\overline Z$ projective. 
\end{lem}
\begin{proof}
Let $\overline B\subset \overline Z$ be any compactification of $B\subset Z$. 
By blowing up $\overline Z$ inside $\overline Z\setminus Z$, 
we may assume that $f:Z\to X$ extends to $\overline f: 
\overline Z\to \overline X$, $\overline Z$ is a simple normal crossing 
variety, and $\overline Z\setminus Z$ is of pure codimension one (see 
\cite[Theorem 1.5]{bierstone-milman}). 
By \cite[Theorem 1.2]{bierstone-p}, we can construct a 
desired compactification. Note that 
we can make $\overline B$ a $\mathbb K$-Cartier $\mathbb K$-divisor 
by the argument in \cite[Section 8]{bierstone-p}. 
\end{proof}

Theorem \ref{8} below is our main theorem of this paper, which is a 
generalization of Koll\'ar's torsion-free and vanishing theorem (see 
\cite[Theorem 2.1]{kollar1}). 
The reader find various applications of Theorem \ref{8} 
in \cite{book}, \cite{fuji-fuji}, and \cite{fujino-slc}. We note that 
Theorem \ref{8} for {\em{embedded normal crossing pairs}} 
was first formulated by Florin Ambro for his theory of {\em{quasi-log varieties}} (cf.~\cite{ambro}). 
For the details of the theory of quasi-log varieties, see \cite[Chapter 3]{book} and \cite{fujino-qlog}. 

\begin{thm}\label{8}
Let $(Y, \Delta)$ be a simple normal crossing 
pair such that $\Delta$ is a boundary $\mathbb R$-divisor on $Y$. 
Let $f:Y\to X$ be a proper morphism to an algebraic variety $X$ and let $L$ be a Cartier 
divisor on $Y$ such that 
$L-(K_Y+\Delta)$ is $f$-semi-ample. 
\begin{itemize}
\item[(i)] every associated prime of $R^qf_*\mathcal O_Y(L)$ is the 
generic point of the $f$-image of some stratum of $(Y, \Delta)$. 
\item[(ii)] 
let $\pi:X\to V$ be a projective morphism to an algebraic variety $V$ such 
that $$L-(K_Y+\Delta)\sim _{\mathbb R}f^*H$$ for 
some $\pi$-ample $\mathbb R$-Cartier 
$\mathbb R$-divisor $H$ on $X$. 
Then $R^qf_*\mathcal O_Y(L)$ is $\pi_*$-acyclic, that is, 
$$R^p\pi_*R^qf_*\mathcal O_Y(L)=0$$ for every $p>0$ and $q\geq 0$. 
\end{itemize}
\end{thm}

\begin{proof}
We set $S=\lfloor \Delta\rfloor$, $B=\{\Delta\}$, 
and $H'\sim _{\mathbb R}L-(K_Y+\Delta)$ throughout this proof. 
Let us start with the proof of (i).  

\setcounter{step}{0}
\begin{step}
First, we assume that $X$ is projective. 
We may assume that $H'$ is semi-ample by replacing 
$L$ (resp.~$H'$) with 
$L+f^*A'$ (resp.~$H'+f^*A'$), where 
$A'$ is a very ample Cartier divisor on $X$. 
Suppose that 
$R^qf_*\mathcal O_Y(L)$ has a local section whose support 
does not contain the $f$-images of any strata of 
$(Y, S+B)$. 
More precisely, let $U$ be a non-empty Zariski open set 
and let $s\in \Gamma (U, R^qf_*\mathcal O_Y(L))$ be a non-zero section 
of $R^qf_*\mathcal O_Y(L)$ on $U$ whose support 
$V\subset U$ does not contain the $f$-images of any strata of $(Y, S+B)$. Let $\overline V$ be the closure 
of $V$ in $X$. 
We note that $\overline V\setminus V$ may contain the $f$-image of some stratum of $(Y, S+B)$. 
Let $Y_2$ be the union of the irreducible 
components of $Y$ that are mapped into $\overline V\setminus V$ and 
let $Y_1$ be the union of the other irreducible components of $Y$. 
We set 
$$
K_{Y_1}+S_1+B_1=(K_Y+S+B)|_{Y_1}
$$ 
such that $S_1$ is reduced and that $\lfloor B_1\rfloor =0$. 
By replacing $Y$, $S$, $B$, $L$, and $H'$ with 
$Y_1$, $S_1$, $B_1$, $L|_{Y_1}$, and $H'|_{Y_1}$, 
we may assume that no irreducible components of $Y$ are mapped into 
$\overline V\setminus V$. 
Let $C$ be a stratum of $(Y, S+B)$ that 
is mapped into $\overline V\setminus V$. 
Let $\sigma:Y'\to Y$ be the blowing-up along $C$. 
We set $E=\sigma^{-1}(C)$. 
We can write 
$$K_{Y'}+S'+B'=\sigma^*(K_Y+S+B)$$ such that 
$S'$ is reduced and $\lfloor B'\rfloor =0$. 
Thus, 
$$
\sigma ^*H'\sim _{\mathbb R}\sigma^*L-(K_{Y'}+S'+B')
$$ 
and 
$$
\sigma ^*H'\sim _{\mathbb R}\sigma ^*L-E-(K_{Y'}+(S'-E)+B'). 
$$
We note that $S'-E$ is effective. 
We replace $Y$, $H'$, $L$, $S$, and $B$ with 
$Y'$, $\sigma ^*H'$, $\sigma^*L-E$, $S'-
E$, and $B'$. 
By repeating this process finitely many times, we may 
assume that $\overline V$ does not contain the $f$-images of any strata of 
$(Y, S+B)$. 
Then 
we can find a very ample Cartier divisor 
$A$ on $X$ with the following properties. 
\begin{itemize}
\item[(a)] $f^*A$ is permissible with respect to 
$(Y, S+B)$, and 
\item[(b)] $R^qf_*\mathcal O_Y(L)\to R^qf_*\mathcal O_Y(L)\otimes 
\mathcal O_X(A)$ is not injective. 
\end{itemize}
We may assume that $H'-f^*A$ is semi-ample by replacing 
$L$ (resp.~$H'$) with 
$L+f^*A$ (resp.~$H'+f^*A$). 
If necessary, we replace $L$ (resp.~$H'$) with 
$L+f^*A''$ (resp.~$H'+f^*A''$), where $A''$ is a very ample 
Cartier divisor. 
Then, we have $$H^0(X, R^qf_*\mathcal O_Y(L))\simeq 
H^q(Y, \mathcal O_Y(L))$$ and 
$$H^0(X, R^qf_*\mathcal O_Y(L)\otimes \mathcal O_X(A))\simeq 
H^q(Y, \mathcal O_Y(L+f^*A)). $$ 
We obtain that 
$$H^0(X, R^qf_*\mathcal O_Y(L))
\to H^0(X, R^qf_*\mathcal O_Y(L)\otimes \mathcal O_X(A))
$$ 
is not injective by (b) if $A''$ is sufficiently ample. 
So, 
$$H^q(Y, \mathcal O_Y(L))
\to 
H^q(Y, \mathcal O_Y(L+f^*A))$$ is not injective. 
It contradicts Theorem \ref{5.1}. 
Therefore, the support of 
every non-zero local section of $R^qf_*\mathcal O_Y(L)$ contains the $f$-image of 
some stratum of $(Y, \Delta)$, equivalently, the support of every non-zero local section 
of $R^qf_*\mathcal O_Y(L)$ is equal to the union of the $f$-images of 
some strata of $(Y, \Delta)$. This means that 
every associated prime of $R^qf_*\mathcal O_Y(L)$ is the generic 
point of the $f$-image of some stratum of $(Y, \Delta)$. 
We finish the proof when $X$ is projective. 
\end{step}
\begin{step}\label{8-2}
Next, we assume that $X$ is not projective. 
Note that the problem is local. So, we shrink $X$ and may assume that 
$X$ is affine. 
We can write $H'\sim _{\mathbb R}H'_1+H'_2$, 
where $H'_1$ (resp.~$H'_2$) is a semi-ample $\mathbb Q$-Cartier $\mathbb Q$-divisor 
(resp.~a semi-ample $\mathbb R$-Cartier $\mathbb R$-divisor) on $Y$. 
We can write $H'_2\sim _{\mathbb R}\sum _ia_iA_i$, 
where $0<a_i<1$ and $A_i$ is a general effective Cartier divisor 
on $Y$ for every $i$. 
Replacing $B$ (resp.~$H'$) with 
$B+\sum _i a_iA_i$ (resp.~$H'_1)$, 
we may assume that $H'$ is a semi-ample $\mathbb Q$-Cartier 
$\mathbb Q$-divisor. 
Without loss of generality, we may further assume that 
$(Y, B+S+H')$ is a simple normal crossing pair.  
We compactify $X$ and apply Lemma \ref{comp}. 
Then we obtain a compactification $\overline f: \overline Y\to \overline X$ of 
$f:Y\to X$. 
Let $\overline {H'}$ be the closure of $H'$ on $\overline Y$. If $\overline {H'}$ 
is not a semi-ample $\mathbb Q$-Cartier 
$\mathbb Q$-divisor, then we take blowing-ups of $\overline Y$ inside $\overline Y\setminus Y$ 
and obtain a semi-ample $\mathbb Q$-Cartier $\mathbb Q$-divisor 
$\widetilde {H'}$ on 
$\overline Y$ such that $\widetilde {H'}|_{Y}=H'$. 
Let $\overline B$ (resp.~$\overline S$) be the 
closure of $B$ (resp.~$S$) on $\overline Y$. 
We may assume that $\overline S$ is Cartier and $\overline B$ is $\mathbb R$-Cartier 
(cf.~Lemma \ref{comp}). 
We construct a coherent sheaf 
$\mathcal F$ on $\overline {Y}$ 
which is an extension of $\mathcal O_Y(L)$. 
We consider Grothendieck's Quot scheme 
$\mathrm{Quot}^{1, \mathcal O_{\overline {Y}}}
_{\mathcal F/\overline{Y}/\overline{Y}}$ (see Theorem \ref{quot-s}). 
Note that the restriction of $\mathrm{Quot}^{1, \mathcal O_{\overline {Y}}}
_{\mathcal F/\overline{Y}/\overline{Y}}$ to $Y$ is nothing but 
$Y$ because $\mathcal F|_Y=\mathcal O_Y(L)$ is a line bundle on $Y$. 
Therefore, the structure morphism from 
$\mathrm{Quot}^{1, \mathcal O_{\overline {Y}}}
_{\mathcal F/\overline{Y}/\overline{Y}}$ to 
$\overline {Y}$ has a section $s$ over $Y$. 
By taking the closure of $s(Y)$ in 
$\mathrm{Quot}^{1, \mathcal O_{\overline {Y}}}
_{\mathcal F/\overline{Y}/\overline{Y}}$, 
we have a compactification $Y^{\dag}$ of 
$Y$ and a line bundle $\mathcal L$ on $Y^{\dag}$ with 
$\mathcal L|_Y=\mathcal O_Y(L)$. 
If necessary, we take more blowing-ups of $Y^{\dag}$ outside $Y$ (cf.~\cite[Theorem 1.2]{bierstone-p}). 
Then we obtain a new compactification $\overline {Y}$ and a 
Cartier divisor $\overline {L}$ on $\overline {Y}$ with $\overline {L}|_Y=L$ 
(cf.~Lemma \ref{comp}). 
In this situation, $\widetilde {H'}\sim _{\mathbb R}\overline L-(K_{\overline Y}
+\overline S +\overline B)$ does not necessarily hold. 
We can write $$H'+\sum _i b_i (f_i)=L-(K_Y+S+B), $$ 
where $b_i$ is a real number and $f_i\in \Gamma(Y, \mathcal K_Y^*)$ for 
every $i$. 
We set $$E=\widetilde {H'}+\sum _i b_i (f_i)-
(\overline L-(K_{\overline Y}+\overline S+
\overline B)).$$   
We note that we can see $f_i\in \Gamma (\overline Y, \mathcal K_{\overline Y}^*)$ for every $i$ 
(cf.~\cite[Section 7.1 Proposition 1.15]{liu}). 
We replace $\overline L$ (resp.~$\overline B)$ with 
$\overline L+\lceil E\rceil$ (resp.~$\overline B+\{-E\}$). 
Then we obtain the desired property of 
$R^q\overline f_*\mathcal O_{\overline Y}
(\overline L)$ since $\overline X$ is projective. We 
note that $\lceil E\rceil$ is Cartier because 
$\Supp E$ is in $\overline Y\setminus Y$ 
and $E$ is $\mathbb R$-Cartier (cf.~Remark \ref{211ne}). 
So, we finish the whole proof of (i). 
\end{step}

From now on, we prove (ii). 

\setcounter{step}{0}
\begin{step} 
We assume that $\dim V=0$. 
In this case, we can write $H\sim _{\mathbb R}H_1+H_2$, 
where $H_1$ (resp.~$H_2$) is an ample $\mathbb Q$-Cartier 
$\mathbb Q$-divisor 
(resp.~an ample $\mathbb R$-Cartier $\mathbb R$-divisor) on $X$. 
So, we can write $H_2\sim _{\mathbb R}\sum _i a_i A_i$, 
where $0<a_i<1$ and $A_i$ is a general very ample Cartier 
divisor on $X$ for every $i$. 
Replacing $B$ (resp.~$H$) with $B+\sum _i a_i f^*A_i$ 
(resp.~$H_1$), we may assume that 
$H$ is an ample $\mathbb Q$-Cartier $\mathbb Q$-divisor. 
We take a general member $A\in |mH|$, 
where $m$ is a sufficiently large and divisible integer, such 
that $A'=f^*A$ and 
$R^qf_*\mathcal O_Y(L+A')$ is $\pi_*$-acyclic, that is, 
$\Gamma$-acyclic, for all $q$. 
By (i), we have the following short exact sequences, 
$$
0\to R^qf_*\mathcal O_Y(L)\to 
R^qf_*\mathcal O_Y(L+A')
\to R^qf_*\mathcal O_{A'}(L+A')\to 0. 
$$ 
for every $q$. 
Note that $R^qf_*\mathcal O_{A'}(L+A')$ is $\pi_*$-acyclic by 
induction on $\dim X$ and $R^qf_*\mathcal O_Y(L+A')$ is also 
$\pi_*$-acyclic by the above assumption. 
Thus, $E^{pq}_2=0$ for $p\geq 2$ in the 
following commutative diagram of spectral sequences. 
$$
\xymatrix{
&E^{pq}_2=H^p(X, R^qf_*\mathcal O_Y(L)) \ar[d]_{\varphi^{pq}}
\ar@{=>}[r]&H^{p+q}(Y, 
\mathcal O_Y(L))\ar[d]_{\varphi^{p+q}}\\ 
&\overline{E}^{pq}_2=H^p(X, R^qf_*\mathcal O_Y(L+A')) 
\ar@{=>}[r]& H^{p+q}(Y, 
\mathcal O_Y(L+A'))
}
$$
We note that $\varphi^{1+q}$ is injective by Theorem \ref{5.1}. 
We have that 
$$E^{1q}_2\overset {\alpha}{\longrightarrow} H^{1+q}(Y, \mathcal O_Y(L))$$ 
is injective by the fact that 
$E^{pq}_2=0$ for $p\geq 2$. 
We also have that $\overline 
{E}^{1q}_2=0$ by the above assumption. 
Therefore, we obtain 
$E^{1q}_2=0$ since the injection 
$$E^{1q}_2\overset{\alpha}{\longrightarrow} H^{1+q}(Y, \mathcal O_Y(L))
\overset{\varphi^{1+q}}{\longrightarrow} 
H^{1+q}(Y, \mathcal O_Y(L+A'))$$ factors 
through $\overline {E}^{1q}_2=0$. 
This implies that $H^p(X, R^qf_*\mathcal O_Y(L))=0$ for every $p>0$. 
\end{step}

\begin{step}\label{o2}  
We assume that $V$ is projective. 
By replacing $H$ (resp.~$L$) with 
$H+\pi^*G$ (resp.~$L+(\pi\circ f)^*G$), where 
$G$ is a very ample Cartier divisor on $V$, 
we may assume that $H$ is an ample $\mathbb R$-Cartier 
$\mathbb R$-divisor. 
If $G$ is a 
sufficiently ample Cartier 
divisor on $V$, then we have $$H^k(V, R^p\pi_*R^qf_*\mathcal O_Y(L)\otimes G)
=0$$ for 
every $k\geq 1$, 
\begin{align*}
H^0(V, R^p\pi_*R^qf_*\mathcal O_Y(L)\otimes \mathcal O_V(G))
&\simeq H^p(X, R^qf_*\mathcal O_Y(L)\otimes \mathcal O_X(\pi^*G))
\\ &\simeq H^p(X, R^qf_*\mathcal O_Y(L+f^*\pi^*G)) 
\end{align*} 
for every $p$ and $q$, 
and $R^p\pi_*R^qf_*\mathcal O_Y(L)\otimes \mathcal O_V(G)$ is generated by its 
global sections for every $p$ and $q$. 
Since 
\begin{align*}
L+f^*\pi^*G-(K_Y+\Delta)\sim _{\mathbb R}f^*(H+\pi^*G), 
\end{align*} 
and $H+\pi^*G$ is ample, 
we can apply Step 1 and obtain $$H^p(X, R^qf_*\mathcal O_Y(L+
f^*\pi^*G))=0$$ for every $p>0$. 
Thus, $R^p\pi_*R^qf_*\mathcal 
O_Y(L)=0$ for every $p>0$ by the above arguments.
\end{step}
\begin{step}\label{o3}  
When $V$ is not projective, we shrink $V$ and may assume 
that $V$ is affine. 
By the same argument as in Step 1, we may assume that $H$ is $\mathbb Q$-Cartier. 
Let $\overline {\pi}:\overline X\to \overline V$ be a compactification of 
$\pi:X\to V$ such that 
$\overline X$ and $\overline V$ are projective. 
We may assume that 
there exists a $\overline \pi$-ample 
$\mathbb Q$-Cartier $\mathbb Q$-divisor 
$\overline H$ on $\overline X$ such that 
$\overline H|_X=H$. 
By Lemma \ref{comp}, we can compactify 
$f:(Y, S+B)\to X$ and obtain $\overline f: (\overline Y, \overline S+\overline B)\to \overline X$. 
We note that 
${\overline f}^*\overline {H}
\sim _{\mathbb R}\overline L-(K_{\overline Y}+\overline S+\overline B)$ does not necessarily hold, 
where $\overline L$ is a Cartier divisor on $\overline Y$ constructed as in Step \ref{8-2} in the proof of (i). 
By the same argument as in Step \ref{8-2} in the proof of (i), 
we obtain that $R^p\pi_*R^qf_*\mathcal O_Y(L)=0$ for every $p>0$ and $q\geq 0$. 
\end{step}
We finish the whole proof of (ii). 
\end{proof}

\section{Semi divisorial log terminal pairs}\label{sec-sdlt}

Let us start with the definition of {\em{semi divisorial log terminal pairs}} in the 
sense of Koll\'ar. For details of singularities which appear in the minimal model 
program, see \cite{fujino-what} and \cite{kollar-book}. 

\begin{defn}[Semi divisorial log terminal pairs]\label{def-sdlt} 
Let $X$ be a pure-dimensional reduced $S_2$ scheme which is 
simple normal crossing in codimension one. 
Let $\Delta=\sum _i a_i\Delta_i$ be an $\mathbb R$-Weil divisor on $X$ such that 
$0<a_i\leq 1$ for every $i$ and that $\Delta_i$ is not contained in the singular locus of $X$, where 
$\Delta_i$ is a prime divisor on $X$ for every 
$i$ and $\Delta_i\ne \Delta_j$ for $i\ne j$. 
Assume that $K_X+\Delta$ is $\mathbb R$-Cartier. 
The pair $(X, \Delta)$ is {\em{semi divisorial log terminal}} 
({\em{sdlt}}, for short) if $a(E, X, \Delta)>-1$ 
for every exceptional divisor $E$ over $X$ 
such that $(X, \Delta)$ is not a simple normal crossing pair at the generic point 
of $c_X(E)$, where 
$c_X(E)$ is the center of $E$ on $X$. 
\end{defn}

We note that if $(X, \Delta)$ is sdlt and 
$X$ is irreducible then $(X, \Delta)$ is a divisorial log terminal pair (dlt, for short). 
The following theorem is a direct generalization 
of \cite[Theorem 4.14]{book} (cf.~\cite[Proposition 2.4]{fujino-lc}). 
It is an easy application of Lemma \ref{rf}. 

\begin{thm}[{cf.~\cite[Theorem 5.2]{fujino-lc}}]\label{sdlt}
Let $(X, D)$ be a semi divisorial log terminal pair. 
Let $X=\bigcup _{i\in I}X_i$ be the irreducible decomposition. 
We set $$Y=\underset{i\in J}\bigcup X_i\subset X$$ for $J\subset I$. 
Then $Y$ is Cohen--Macaulay, semi-normal, and has only Du Bois singularities. 
In particular, each irreducible component of $X$ is normal and $X$ itself is Cohen--Macaulay. 
\end{thm}
We note that an irreducible component of a semi-normal scheme need not 
be semi-normal (see \cite[Example 9.8]{kollar-book}). 
We also note that an irreducible component of a Cohen--Macaulay scheme need not be Cohen--Macaulay. 
The author learned the following example from Shunsuke Takagi. 

\begin{ex} We set 
$$
R=\mathbb C[x, y, z, w]/(yz-xw, xz^2-y^2w). 
$$ 
Then $X=\Spec R$ is a reduced reducible two-dimensional 
Cohen--Macaulay scheme. 
An irreducible component 
$$
Y=\Spec R/(y^3-x^2z, z^3-yw^2)
$$ 
of $X$ is not Cohen--Macaulay. 
It is because 
$$
R/(y^3-x^2z, z^3-yw^2)\simeq \mathbb C[s^4, s^3t, st^3, t^4]. 
$$
\end{ex}

The Cohen--Macaulayness of $X$ is very important 
for various duality theorems. We use it in the proof of Theorem \ref{11} in \cite{fuji-fuji}. 

Let us start the proof of Theorem \ref{sdlt}. 

\begin{proof}[Proof of {\em{Theorem \ref{sdlt}}}]
By \cite[Theorem 1.2]{bierstone-p}, there is a morphism $f:Z\to X$ given by a composite of blowing-ups with smooth centers such that $(Z, f^{-1}_*D+\Exc (f))$ is a simple normal crossing 
pair and that $f$ is an isomorphism over $U$, where 
$U$ is the largest Zariski open subset of $X$ such that $(U, D|_U)$ is a simple normal crossing pair. 
Then we can write 
$$
K_Z+D'=f^*(K_X+D)+E, 
$$ 
where $D'$ and $E$ are effective and have no common irreducible components. 
By construction, $E$ is $f$-exceptional and $\Supp (E+D')$ is a simple normal crossing divisor on $Z$. 
Since $X$ is $S_2$ and simple normal crossing in codimension one, $X$ is semi-normal. 
Then we obtain $f_*\mathcal O_Z\simeq \mathcal O_X$. 
Let $Z=\bigcup _{i\in I}Z_i$ be the irreducible decomposition. 
We consider the short exact sequence
$$
0\to \mathcal O_V(-W|_V)\to \mathcal O_Z\to \mathcal O_W\to 0, 
$$ 
where $W=\bigcup _{i\in J}Z_i$ and $V=\bigcup _{i\in I\setminus J}Z_i$. 
Therefore, 
$$
0\to \mathcal O_V(\lceil E\rceil -W|_V)\to \mathcal O_Z(\lceil E\rceil)\to \mathcal O_W
(\lceil E\rceil)\to 0
$$ 
is exact. 
We note that 
$\lceil E\rceil$ is Cartier (cf.~Remark \ref{211ne}). 
By Lemma \ref{rf}, 
$R^if_*\mathcal O_Z(\lceil E\rceil)=0$ for every $i>0$. 
We note that 
$$
\lceil E\rceil\sim _{\mathbb R, f}K_Z+D'+\{-E\}. 
$$ 
Since 
$$
(\lceil E\rceil -W)|_V\sim _{\mathbb R, f}K_V+(D'+\{-E\})|_V, 
$$ 
$R^if_*\mathcal O_V(\lceil E\rceil -W|_V)=0$ for every $i>0$ 
by Lemma \ref{rf} again. 
Therefore, we obtain 
that 
$$
0\to f_*\mathcal O_V(\lceil E\rceil -W|_V)\to f_*\mathcal O_Z(\lceil E\rceil)\simeq \mathcal 
O_X\to f_*\mathcal O_W(\lceil E\rceil)\to 0 
$$ is exact and that 
$R^if_*\mathcal O_W(\lceil E\rceil)=0$ for 
every $i>0$. 
Since $\lceil E\rceil |_W$ is effective 
and $\mathcal O_X\to f_*\mathcal O_W(\lceil E\rceil)\to 0$ factors 
through $\mathcal O_Y$, 
we have $\mathcal O_Y\simeq f_*\mathcal O_W\simeq f_*\mathcal O_W(\lceil E\rceil)$. 
Therefore, $Y$ is semi-normal because so is $W$. 
In the derived category of coherent sheaves on $Y$, 
the composition 
\begin{equation}\label{siki1}
\mathcal O_Y\to Rf_*\mathcal O_W\to Rf_*\mathcal O_W(\lceil E\rceil )\simeq \mathcal O_Y
\end{equation} 
is a quasi-isomorphism. 
Therefore, $Y$ has only Du Bois singularities because $W$ is a simple normal crossing 
variety. 
On the other hand, 
$R^if_*\omega_W=0$ for every $i>0$ by Lemma \ref{rf}. 
By applying the Grothendieck duality to (\ref{siki1}): 
$$
\mathcal O_Y\to Rf_*\mathcal O_W\to \mathcal O_Y, 
$$
we obtain 
\begin{equation}\label{siki2}
\omega^\bullet _Y\overset {a}\to Rf_*\omega^\bullet _W\overset {b}\to \omega^\bullet _Y,  
\end{equation} 
where $\omega^\bullet_Y$ (resp.~$\omega^\bullet_W$) is 
the dualizing complex of $Y$ (resp.~$W$). 
Note that $b\circ a$ is a quasi-isomorphism. 
Thus we have 
$$
h^i(\omega^\bullet _Y)\subset R^if_*\omega^\bullet _W=R^{i+d}f_*\omega_W
$$ 
where $d=\dim Y=\dim W$. 
This implies that $h^i(\omega^\bullet _Y)=0$ for every $i>-\dim Y$. 
Thus, $Y$ is Cohen--Macaulay and $\omega^\bullet _Y\simeq \omega_Y[d]$. 
\end{proof}

As a byproduct of the proof of Theorem \ref{sdlt}, we obtain the following useful 
vanishing theorem. Roughly speaking, Proposition \ref{prop42} says that 
$Y$ has only {\em{semi-rational}} singularities. 

\begin{prop}\label{prop42}
In the notation of the proof of {\em{Theorem \ref{sdlt}}}, 
$f_*\mathcal O_W\simeq \mathcal O_Y$ and $R^if_*\mathcal O_W=0$ for every $i>0$. 
\end{prop}
\begin{proof}
By (\ref{siki2}) in the proof of Theorem \ref{sdlt}, 
we obtain 
$$
\omega_Y\overset{\alpha}\to f_*\omega_W\overset{\beta}\to \omega_Y 
$$ 
where $\beta\circ \alpha$ is an isomorphism. 
Since $\omega_W$ is locally free and $f$ is an isomorphism over $U$, 
$f_*\omega_W$ is a pure sheaf of dimension $d$. 
Thus $f_*\omega_W\simeq \omega_Y$ because 
they are isomorphic over $U$. 
Then we obtain $Rf_*\omega^\bullet _W\simeq \omega^\bullet _Y$ 
in the derived category of coherent sheaves on $Y$. By the Grothendieck 
duality, 
$Rf_*\mathcal O_W\simeq R\mathcal H om (Rf_*\omega^\bullet _W, 
\omega^\bullet _Y)\simeq \mathcal O_Y$ in the derived category of 
coherent sheaves on $Y$. Therefore, 
$f_*\mathcal O_W\simeq \mathcal O_Y$ and $R^if_*\mathcal O_W=0$ for 
every $i>0$. 
\end{proof}

As an easy application of Theorem \ref{sdlt}, we have 
an adjunction formula for sdlt pairs. 

\begin{cor}[Adjunction for sdlt pairs]\label{adj}
In the notation of {\em{Theorem \ref{sdlt}}}, 
we define $D_Y$ by 
$$
(K_X+D)|_Y=K_Y+D_Y. 
$$ 
Then the pair $(Y, D_Y)$ is semi divisorial log terminal. 
\end{cor}
\begin{proof}
By Theorem \ref{sdlt}, $Y$ is Cohen--Macaulay. 
In particular, $Y$ satisfies Serre's $S_2$ condition. 
Then it is easy to see that the pair $(Y, D_Y)$ is semi divisorial log terminal. 
\end{proof}

We close this section with an important remark. 

\begin{rem}
Let $(X, D)$ be a semi divisorial log terminal pair in the sense of 
Koll\'ar (see Definition \ref{def-sdlt}). 
Then it is a semi divisorial log terminal pair in the sense of \cite[Definition 1.1]{fujino-abun}. 
A key point is that any irreducible component of $X$ is normal (see Theorem \ref{sdlt}). 
When the author defined semi divisorial log terminal pairs 
in \cite[Definition 1.1]{fujino-abun}, the theory of resolution of 
singularities for {\em{reducible}} varieties (cf.~\cite{bierstone-milman} and \cite{bierstone-p}) 
was not available. 
\end{rem}

\section{Semi-positivity theorem}\label{sec-semi}

In \cite[Chapter 2]{book}, we discuss mixed Hodge structures on cohomology 
groups with compact support for the proof of Theorem \ref{2} (see also \cite{fujino-inj}). 
In \cite{fuji-fuji}, we investigate variations of 
mixed Hodge structures on cohomology groups with compact support. 
By the Hodge theoretic semi-positivity theorem obtained in \cite[Section 6]{fuji-fuji}, 
we can prove the following theorem as an application of 
Theorem \ref{8}. 

\begin{thm}[Semi-positivity theorem]\label{11}
Let $(X, D)$ be a simple normal crossing pair such that 
$D$ is reduced and let $f:X\to Y$ be 
a projective surjective morphism onto a smooth 
complete algebraic variety $Y$. 
Assume that every stratum of $(X,D)$ is dominant onto $Y$. 
Let $\Sigma$ be a simple normal crossing divisor on $Y$ such that 
every stratum of $(X, D)$ is smooth over ${Y}_0=Y\setminus \Sigma$.
Then $R^if_*\omega_{X/Y}(D)$ is locally free for every $i$. 
We set $X_0=f^{-1}(Y_0)$,
$D_0=D|_{X_0}$, and $d=\dim X-\dim Y$. 
We further assume that 
all the local monodromies on
$R^{d-i}({f|_{X_0\setminus D_0}})_{!}\mathbb Q_{{X_0} \setminus D_0}$ 
around $\Sigma$ 
are unipotent. 
Then we obtain that
$R^if_*\omega_{X/Y}(D)$ is 
a semi-positive {\em{(}}in the sense of Fujita--Kawamata{\em{)}} 
locally free sheaf on $Y$, that is, a {\em{nef}} locally free sheaf on $Y$. 
\end{thm}

We note that Theorem \ref{11} is a generalization of the Fujita--Kawamata semi-positivity theorem 
(cf.~\cite{kawamata1}).  We also note that 
Theorem \ref{11} contains the main theorem of \cite{high}. 
In \cite{high}, we use variations of mixed Hodge structures on cohomology groups of smooth quasi-projective 
varieties. However, our formulation in \cite{fuji-fuji} based on 
mixed Hodge  structures on cohomology groups with compact support is more suitable for reducible varieties 
than the formulation in \cite{high} (cf.~\ref{hodge}). 
Theorem \ref{8} and Theorem \ref{11} show that 
the theory of mixed Hodge structures on cohomology groups 
with compact support is useful for the study of 
higher dimensional algebraic varieties. 
For details, see \cite[Chapter 2]{book}, 
\cite{fuji-fuji}, and \cite{fujino-inj}.  

Finally, we note that in \cite{fujino-semi} we prove the projectivity of the moduli spaces of stable varieties 
as an application of Theorem \ref{11} by Koll\'ar's projectivity criterion. 

\end{document}